\documentclass[12pt]{article}
\usepackage{amssymb,amsmath,amsthm}

\def\soejia{\alpha}
\def\soejib{\beta}
\def\soejic{\gamma}
\def\soejid{\delta}

\newtheorem{Theorem}{Theorem}[section]
\newtheorem{Lemma}[Theorem]{Lemma}

\newtheorem{Example}[Theorem]{Example}

\def\tb{{\bold t}}

\def\0b{{\bold 0}}
\def\Ac{{\cal A}}

\def\tb{{\bold t}}

\def\0b{{\bold 0}}
\def\Ac{{\cal A}}

\newcommand{\totalitems}{\tau}

\def\AAA{{{
A}_{\totalitems, {\bf b}, {\bf c}, {\bf r}, {\bf s}}}}
\def\Ac{{
A}}     

\def\Ac{{
A}}

\usepackage{bm}
\bmdefine{\Bzero}{0}
\bmdefine{\Bone}{1}

\def\Bone{{\bf 1}}

\def\tb{{\bf t}}

\def\RR{{\mathbb R}}
\def\ZZ{{\mathbb Z}}

\title{Gr\"obner bases of nested configurations}
\author{Satoshi Aoki, Takayuki Hibi, Hidefumi Ohsugi and Akimichi Takemura}
\date{}

\begin{document}
\maketitle

\begin{abstract}
In this paper we introduce a new and large family of configurations whose
toric ideals possess quadratic Gr\"obner bases. As an application, a
generalization of algebras of Segre--Veronese type will be studied.
\end{abstract}

\section{Introduction}

Let $K[\tb]=K[t_1, \ldots, t_d]$ be the polynomial ring 
over a field $K$.
A finite set $A$ of monomials of $K[\tb]$
is called a {\em configuration} of $K[\tb]$
if there exists a nonnegative vector $(w_1,\ldots,w_d) \in \RR_{\geq 0}^d$ such that
$\sum_{i=1}^d w_i a_i = 1$ for all $t_1^{a_1} \cdots t_d^{a_d} \in A$.
Let $A$ be a configuration of $K[\tb]$.
We associate $A$
with the homogeneous semigroup ring
$K[A]$ which is
the subalgebra of $K[\tb]$ generated by
the monomials of $\Ac$.
Let $K[X] = K[\{x_{M} \, | \, M \in \Ac\}]$ denote
the polynomial ring over $K$ in the variables 
$x_{M}$ with $M \in \Ac$, where each $\deg(x_M) = 1$.
The {\em toric ideal} $I_\Ac$ of $\Ac$ 
is the kernel of the surjective homomorphism 
$\pi \, : \, K[X] \to K[A]$
defined by setting
$\pi(x_{M}) = M  $ for all $M \in \Ac$.
It is known that the toric ideal
$I_\Ac$ is generated by the binomials $u - v$, where
$u$ and $v$ are monomials of $K[X]$, with
$\pi(u) = \pi(v)$.
Moreover, since $A$ is a configuration,
$I_A$ is generated by homogeneous binomials.
See, e.g., \cite[Section 4]{Stu}.

A fundamental question in commutative algebra
is to determine whether $K[A]$ is {\em Koszul}.
A Gr\"obner basis ${\cal G}$ is called a {\it quadratic Gr\"obner basis}
if ${\cal G}$ consists of quadratic homogeneous polynomials.
Even though it is difficult to prove that
$K[A]$ is Koszul, the hierarchy 
(i) $\Longrightarrow$ (ii) $\Longrightarrow$ (iii)
is known
among the following properties:  
\begin{enumerate}
\item[(i)]
$I_\Ac$ possesses a quadratic Gr\"obner basis.
\item[(ii)]
$K[A]$ is Koszul;
\item[(iii)]
$I_\Ac$ is generated by quadratic binomials.
\end{enumerate}
However both (ii) $\Longrightarrow$ (i)
and (iii) $\Longrightarrow$ (ii) are false in general.
One can find counterexamples for them in \cite[Examples 2.1 and 2.2]{OH2}.

Let $A$ be a 
configuration of
a polynomial ring
$K[{\bf t}] = K[t_1,\ldots,t_d]$ with
$d$ variables.
For each $i = 1,2,\ldots,d$, let 
$B_i=\{m_1^{(i)},\ldots,m_{\lambda_i}^{(i)}\}$ be 
a configuration of a polynomial ring 
$K[{\bf u}^{(i)}] = K[ u_1^{(i)},\ldots,u_{\mu_i}^{(i)}]$
with $\mu_i$ variables.
The {\it nested configuration} arising from $A$ and $B_1,...,B_d$
is
the configuration
$$
A(B_1,\ldots,B_d) := 
\left\{ \left.
{m}_{j_1}^{(i_1)} \cdots {m}_{j_r}^{(i_r)} 
\ \right| \ 1 \leq r \in  \ZZ , \ 
t_{i_1} \cdots t_{i_r} \in A, \ \ 
1 \leq j_k \leq \lambda_{i_k} 
\right\}
$$
of a polynomial ring
$K[{\bf u}^{(1)},\ldots,{\bf u}^{(d)} ]$.
The main result of the present paper is as follows:

\bigskip

\noindent
{\bf Theorem \ref{main1}.}
{\it
Work with the same notation as above.
If the toric ideals $I_A$, $I_{B_1}, \ldots , I_{B_d}$ possess quadratic
Gr\"obner bases,
then
the toric ideal $I_{A(B_1,\ldots,B_d) }$
possesses a quadratic Gr\"obner basis.
}

\bigskip

In addition, in Section 4, we study a quadratic Gr\"obner basis of the toric ideal of
$A(B_1,\ldots,B_d)$ where $A$ comes from a Segre--Veronese configuration.

\section{Motivation from statistics}

In this section we present a statistical problem, which motivates
a generalization of Segre-Veronese configuration considered in our
previous paper [AHOT].

We consider nested selection of groups and items from the
groups.  For example, consider an examination on mathematics which
consists of $J$ groups of problems (e.g.\ algebra, geometry and
statistics) and each group $j$ consists of $m_j$ individual
problems.   For simplicity let $J=3$ and $m_j\equiv 3$.
Label the nine individual problems as  A1, A2, A3, G1, G2,
G3 and S1, S2, S3.
Suppose that each examinee is asked to choose two groups of problems
and then $c_j=2$ problems from each chosen group $j$.   Then there are
27 patterns of selections of four problems
as (A1,A2,G1,G2), (A1,A3,G1,G2),
(A2,A3,G1,G2), \dots,(G2,G3,S2,S3).  Now as a simple statistical model
suppose that each problem is chosen according to its own
attractiveness, independent of the choices of other problems within
the same group as well as the choice of other group.  Let
$q_{\rm A1},q_{\rm A2}, \dots, q_{\rm S3}$ denote the attractiveness of each
problem. Then the probabilities of the selections are expressed as
\begin{align*}
&{\rm Prob}({\rm A1,A2,G1,G2})= c\, q_{\rm A1}q_{\rm A2}q_{\rm G1}q_{\rm
  G2}, \\
& \qquad\qquad \dots\\
&{\rm Prob}({\rm G2,G3,S2,S3})= c\, q_{\rm G2} q_{\rm G3}q_{\rm S2}q_{\rm
  S3},
\end{align*}
where $c$ is the normalizing constant so that the 27 probabilities
sum to one.  Now associate a configuration $A$ to the semigroup ring
\[
K(q_{\rm A1}q_{\rm A2}q_{\rm G1}q_{\rm   G1}, \dots, q_{\rm G2} q_{\rm G3}q_{\rm S2}q_{\rm
  S31}), 
\]
which is a subring of the polynomial ring 
$K(q_{\rm A1},q_{\rm A2}, \dots, q_{\rm S3})$ in  nine variables.
A system of generators of the toric ideal for $I_A$, such as the
reduced Gr\"obner basis, is required for statistical test of this model.
This example corresponds to $A(B_1,B_2,B_3)$ where $A=\{t_1t_2, t_1t_3,
t_2t_3\}$ and $B_1$, $B_2$, $B_3$ are copies of $A$ with different variables.

Enumeration of different selections becomes somewhat more complicated
if the same item can be chosen more than once (``sampling with
replacement'').  Suppose that a customer is given two (identical) coupons,
which allow the customer to go to one of several shops and buy
two items at a discount at the  shop. 
Buying the same item twice is allowed.  For
simplicity suppose that there are only two shops A,B and they sell
only two different items \{A1, A2\} and \{B1, B2\}, respectively.  A
person may buy A1 four times, by going to the shop A twice and buying
A1 twice each time.  Or a person may buy each of A1, A2, B1, B2 once.
Note that in this scheme it is not possible to buy three items from
shop A and 1 item from shop B.  Again we can think of a statistical
model that the relative popularity of selections is explained entirely
by the attractiveness of each item.  This corresponds to Example \ref{nestexample}
below, where ${\rm A1}=u_1^{(1)}$, ${\rm A2}=u_2^{(1)}$, 
${\rm B1}=u_1^{(2)}$, ${\rm B2}=u_2^{(2)}$.

Note that in the above examples we can also consider recursive nesting of subgroups.

\section{Nested configurations}

In this section, we introduce
an effective method to construct
semigroup rings whose toric ideals have quadratic Gr\"obner bases.

Let $A$ be a 
configuration of
a polynomial ring
$K[{\bf t}] = K[t_1,\ldots,t_d]$ with
$d$ variables.
For each $i = 1,2,\ldots,d$, let 
$B_i=\{m_1^{(i)},\ldots,m_{\lambda_i}^{(i)}\}$ be 
a configuration of a polynomial ring 
$K[{\bf u}^{(i)}] = K[ u_1^{(i)},\ldots,u_{\mu_i}^{(i)}]$
with $\mu_i$ variables.
The {\it nested configuration} arising from $A$ and $B_1,...,B_d$
is
the configuration
$$
A(B_1,\ldots,B_d) := 
\left\{ \left.
{m}_{j_1}^{(i_1)} \cdots {m}_{j_r}^{(i_r)} 
\ \right| \ 1 \leq r \in  \ZZ , \ 
t_{i_1} \cdots t_{i_r} \in A, \ \ 
1 \leq j_k \leq \lambda_{i_k} 
\right\}
$$
of a polynomial ring
$K[{\bf u}^{(1)},\ldots,{\bf u}^{(d)} ]$.

\begin{Example}
{\em
If $B_i = \left\{m_1^{(i)} \right\}$ for all $1 \leq i \leq d$,
then we have $K[A(B_1,\ldots,B_d)] \simeq K[A]$.
}
\end{Example}

\begin{Example}
{\em
Let $A=
\{
t_1^r
\}
$
and
let
$B_1= \{u_1,\ldots, u_\lambda\}$ be the set of variables.
Then $K[A(B_1)]$ is $r$th Veronese subring of 
the polynomial ring $K[B_1]= K[u_1,\ldots,u_\lambda]$.
}
\end{Example}

\begin{Example}
{\em
Let $A=
\{
t_1 t_2
\}
$
and
let
$B_1= \{u_1^{(1)},\ldots, u_{\lambda_1}^{(1)}\}$
and
$B_2= \{u_1^{(2)},\ldots, u_{\lambda_2}^{(2)}\}$
be the sets of variables.
Then $K[A(B_1,B_2)]$ is the Segre product of 
the polynomial rings
$K[B_1]= K[u_1^{(1)},\ldots, u_{\lambda_1}^{(1)}]$
and
$K[B_2]= K[u_1^{(2)},\ldots, u_{\lambda_2}^{(2)}]$.
}
\end{Example}

Let $\eta$ be the cardinality of 
$ A(B_1,\ldots, B_d)$ and set
$A(B_1,\ldots, B_d)= \{M_1,\ldots,M_\eta \}$.
Let
\begin{eqnarray*}
K[{\bf x}]&=&K[x_{M_1} ,\ldots,x_{M_\eta }]\\
K[{\bf y}] &=& K[\{ y_{i_1 \cdots i_r}\}_{1 \leq r \in  \ZZ , \ i_1 \leq \cdots \leq i_r, \ t_{i_1} \cdots t_{i_r}\in A} ]\\
K\left[{\bf z}^{(i)}\right] &=& K\left[z_1^{(i)},\ldots , z_{\lambda_i}^{(i)}\right] \ \ \ \ \ \ {(i= 1,2,\ldots,d)}
\end{eqnarray*}
be polynomial rings.
The toric ideal $I_{A}$ is the kernel of the
homomorphism
$\pi_0 : K[{\bf y}] \longrightarrow K[{\bf t}]$
defined by setting $\pi_0  (y_{i_1 \cdots i_r})=t_{i_1} \cdots t_{i_r}.$
The toric ideal $I_{B_i}$ is the kernel of the
homomorphism
$\pi_{i} : K[{\bf z}^{(i)}] \longrightarrow K[{\bf u}^{(i)}]$
defined by setting $\pi_{i} (z_j^{(i)}) =m_j^{(i)}.$
The toric ideal $I_{A(B_1,\ldots,B_d) }$ is the kernel of the
homomorphism
$\pi : K[{\bf x}] \longrightarrow K[{\bf u}^{(1)},\ldots,{\bf u}^{(d)} ]$
defined by setting
$\pi\left(x_{M}\right)=M.$

Let ${\cal G}_i$ be a Gr\"obner basis of $I_{B_i}$ with respect to 
a monomial order $<_i$ for $1 \leq i \leq d$.
For each $M \in A(B_1,\ldots,B_d)$,
the expression $M = {m}_{j_1}^{(i_1)} \cdots {m}_{j_r}^{(i_r)}$ is called
{\it standard}
if
$$\prod_{i_k = j , \ \  1 \leq k \leq r} z_{j_k}^{(i_k)} $$
is a standard monomial with respect to ${\cal G}_j$
for all $1 \leq j \leq d$.

\begin{Example}
\label{nestexample}
{\em
Let $A=\{ t_1^2,t_1 t_2, t_2^2\}$,
\begin{eqnarray*}
B_1&=& \left\{ m_1^{(1)} = \left(u_1^{(1)}\right)^2, \ m_2^{(1)} = u_1^{(1)}  u_2^{(1)},  \ 
m_3^{(1)} = \left(u_2^{(1)}\right)^2 \right\},\\
B_2&=& \left\{ m_1^{(2)} = \left(u_1^{(2)}\right)^2, \ m_2^{(2)} =  u_1^{(2)}  u_2^{(2)}, \ 
m_3^{(2)} =  \left(u_2^{(2)}\right)^2 \right\}.
\end{eqnarray*}
Then 
$
A(B_1,B_2)
$
consists of the monomials
$$
 \left(u_1^{(1)}\right)^4, \left(u_1^{(1)}\right)^3  u_2^{(1)} , \left(u_1^{(1)}\right)^2  \left(u_2^{(1)}\right)^2 ,
 u_1^{(1)}  \left(u_2^{(1)}\right)^3 , \left(u_2^{(1)}\right)^4,
$$
$$
\left(u_1^{(1)}\right)^2 \left(u_1^{(2)}\right)^2, \left(u_1^{(1)}\right)^2 u_1^{(2)}  u_2^{(2)}, \left(u_1^{(1)}\right)^2 \left(u_2^{(2)}\right)^2
$$
$$
 u_1^{(1)}  u_2^{(1)} \left(u_1^{(2)}\right)^2,  u_1^{(1)}  u_2^{(1)} u_1^{(2)}  u_2^{(2)},  u_1^{(1)}  u_2^{(1)} \left(u_2^{(2)}\right)^2
$$
$$
 \left(u_2^{(1)}\right)^2 \left(u_1^{(2)}\right)^2,  \left(u_2^{(1)}\right)^2 u_1^{(2)}  u_2^{(2)},   \left(u_2^{(1)}\right)^2
\left(u_2^{(2)}\right)^2
$$
$$
 \left(u_1^{(2)}\right)^4, \left(u_1^{(2)}\right)^3  u_2^{(2)} , \left(u_1^{(2)}\right)^2  \left(u_2^{(2)}\right)^2 ,
 u_1^{(2)}  \left(u_2^{(2)}\right)^3 , \left(u_2^{(2)}\right)^4
$$
and, with respect to any monomial order,
\begin{eqnarray*}
{\cal G}_0 &=& \{ y_{1 1} y_{2 2} - y_{1 2}^2\},\\
{\cal G}_1 &=& \{ z_1^{(1)} z_3^{(1)} -(z_2^{(1)})^2\},\\
{\cal G}_2 &=& \{ z_1^{(2)} z_3^{(2)} -(z_2^{(2)})^2\}
\end{eqnarray*}
are Gr\"obner bases of
\begin{eqnarray*}
I_A &=& \left< y_{1 1} y_{2 2} - y_{1 2}^2 \right>,\\
I_{B_1} &=& \left< z_1^{(1)} z_3^{(1)} -(z_2^{(1)})^2 \right>,\\
I_{B_2} &=& \left< z_1^{(2)} z_3^{(2)} -(z_2^{(2)})^2 \right>,
\end{eqnarray*}
respectively.
Let $>_0$ be a lexicographic order induced by $y_{1 1} >_0 y_{1 2} >_0 y_{2 2}$
and let $>_i$ a lexicographic order induced by
$ z_1^{(i)} >_i z_2^{(i)} >_i z_3^{(i)}$ for $i = 1,2$.
For example, $$M = \left(u_1^{(1)}\right)^2  \left(u_2^{(1)}\right)^2  \in A(B_1,B_2)$$
has two expressions, that is,
$M =  m_1^{(1)} m_3^{(1)}$ and $M =  \left(m_2^{(1)} \right)^2$.
Since $z_1^{(1)} z_3^{(1)}$ is not standard and $(z_2^{(1)})^2$
is standard with respect to ${\cal G}_1$,
$M =  m_1^{(1)} m_3^{(1)}$ is not a standard expression and
$M =  \left(m_2^{(1)} \right)^2$ is a standard expression.
}
\end{Example}

\bigskip

In order to study the relation among $I_{A}$, $I_{B_i}$ and $I_{A(B_1,\ldots,B_d) }$,
we define homomorphisms
\begin{eqnarray*}
\varphi_0 : K[{\bf x}] \longrightarrow K[{\bf y}] &,& 
\varphi_0 \left(x_{{m}_{j_1}^{(i_1)} \cdots {m}_{j_r}^{(i_r)}}\right) = y_{i_1 \cdots i_r}, \\
\varphi_j : K[{\bf x}] \longrightarrow K[{\bf z}^{(j)}] &,&  
\varphi_j \left(x_{{m}_{j_1}^{(i_1)} \cdots {m}_{j_r}^{(i_r)}}\right) = 
\prod_{i_k = j , \ \  1 \leq k \leq r} z_{j_k}^{(i_k)},
\end{eqnarray*}
where ${m}_{j_1}^{(i_1)} \cdots {m}_{j_r}^{(i_r)}$ is the standard expression
defined above.
For example,
$$\varphi_1 \left( x_{\left(u_1^{(1)}\right)^2  \left(u_2^{(1)}\right)^2 } \right)$$
is not $z_1^{(1)} z_3^{(1)}$  but $(z_2^{(1)})^2$
in Example \ref{nestexample}.
Throughout this paper,
we order the monomials of
$A(B_1,\ldots, B_d)$ as
$A(B_1,\ldots, B_d) = \{M_1,\ldots,M_\eta \}$
where
$$
\ \ \ \ \ \ \ \ \ \ \ \ 
\ \ \ \ \ \ \ \ \ \ \ \ 
\pi_0 \circ \varphi_0(x_{M_1}) \geq_{lex} \cdots \geq_{lex} \pi_0 \circ  \varphi_0(x_{M_\eta}) 
\ \ \ \ \ \ \ \ \ \ \ \ \ \ \ \ \ \ \ \ \ \ \ \ (*)$$
with respect to the lexicographic order $<_{lex}$ induced by $t_1 > \cdots > t_d$.

\begin{Lemma}
\label{keylemma}
Let $f$ be a binomial in $K[{\bf x}]$.
Then the following conditions are equivalent:
\begin{enumerate}
\item[{\em (i)}]
$f \in I_{A(B_1,\ldots,B_d) }$;
\item[{\em (ii)}]
$\varphi_i(f) \in I_{B_i}$ for all $1 \leq i \leq d$.
\end{enumerate}
Moreover, if the above conditions hold, then we have
$\varphi_0 (f) \in I_A$.
\end{Lemma}

\begin{proof}
Let $f ={\bf x}^{\bf \alpha} - {\bf x}^{\bf \beta}$. 
Since
$\pi(x_{M})=M
= \prod_{j=1}^d \pi_j \circ \varphi_j (x_{M})$
and each $\pi_j \circ \varphi_j (x_{M})$ belongs to
$K[{\bf u}^{(j)}]$,
we have 
\begin{eqnarray*}
f \in I_{A(B_1,\ldots,B_d) } =\ker(\pi)
&\Longleftrightarrow  &
\pi({\bf x}^{\bf \alpha} )=\pi({\bf x}^{\bf \beta} )\\
&\Longleftrightarrow &
 \prod_{j=1}^d \pi_j \circ \varphi_j ({\bf x}^{\bf \alpha})
=
 \prod_{j=1}^d \pi_j \circ \varphi_j ({\bf x}^{\bf \beta})\\
&\Longleftrightarrow &
\pi_j \circ \varphi_j ({\bf x}^{\bf \alpha})
=
\pi_j \circ \varphi_j ({\bf x}^{\bf \beta})
\ \mbox{ for all } 1 \leq j \leq d\\
&\Longleftrightarrow &
\pi_j \circ \varphi_j (f)
=
0
\ \mbox{ for all } 1 \leq j \leq d\\
&\Longleftrightarrow &
\varphi_j(f) \in \ker (\pi_j)=I_{B_j}
\ \mbox{ for all } 1 \leq j \leq d.
\end{eqnarray*}
Thus (i) $\Longleftrightarrow $ (ii) holds.

Recall that $I_{B_i}$ is homogeneous
for all $1 \leq i \leq d$.
If $M = {m}_{j_1}^{(i_1)} \cdots {m}_{j_r}^{(i_r)}$, then
$$\pi_0 \circ \varphi_0 
(x_{M})=t_{i_1} \cdots t_{i_r}
= \prod_{j=1}^d  t_j^{\deg \left( \varphi_j \left(x_{M}\right) \right)}.$$
Hence we have
\begin{eqnarray*}
& & \varphi_j(f) \in I_{B_j} \ \mbox{ for all } 1 \leq j \leq d\\
&\Longrightarrow &
\deg ( \varphi_j ({\bf x}^{\bf \alpha}) ) = \deg ( \varphi_j ({\bf x}^{\bf \beta}) )
\ \mbox{ for all } 1 \leq j \leq d\\
&\Longrightarrow &
\pi_0 \circ \varphi_0 ({\bf x}^{\bf \alpha})
=
\prod_{j=1}^d  t_j^{\deg ( \varphi_j ({\bf x}^{\bf \alpha}) )}
=
\prod_{j=1}^d  t_j^{\deg ( \varphi_j ({\bf x}^{\bf \beta}) )}
=
\pi_0 \circ \varphi_0 ({\bf x}^{\bf \beta})\\
&\Longrightarrow &
\pi_0 \circ \varphi_0 ( f ) =0\\
&\Longrightarrow &
\varphi_0 ( f ) \in \ker(  \pi_0 )=I_A
\end{eqnarray*}
as desired.
\end{proof}

\begin{Theorem}
\label{main1}
Suppose that the toric ideals $I_A$, $I_{B_1}, \ldots , I_{B_d}$ possess quadratic
Gr\"obner bases ${\cal G}_0$, ${\cal G}_1, \ldots, {\cal G}_d$ 
with respect to $<_0$, $<_1, \ldots, <_d$ respectively.
Then the toric ideal $I_{A(B_1,\ldots,B_d) }$
possesses the reduced
Gr\"obner basis ${\cal G}$ consisting of all binomials of the form
$
x_{M_{\soejia}} 
x_{M_{\soejib}} 
-
x_{M_{\soejic}} 
x_{M_{\soejid}} 
$
where 
$M_{\soejic}= {m}_{j_1}^{(i_1)}{m}_{j_2}^{(i_2)} \cdots {m}_{j_{r}}^{(i_{r})}$,
$M_{\soejid}= {m}_{\ell_1}^{(k_1)} {m}_{\ell_2}^{(k_2)} \cdots {m}_{\ell_{s}}^{(k_{s})}$
with $ \soejic \leq \soejid $
and
\begin{eqnarray}
\varphi_j(x_{M_{\soejia}} 
x_{M_{\soejib}}) &\stackrel{{\cal G}_j}{\longrightarrow } & \varphi_j(x_{M_{\soejic}} 
x_{M_{\soejid}} ) 
\ \ \ \ \ \ \ \ 
\mbox{ for each } j = 0,1,\ldots, d ,\\
i_\lambda  = k_\mu  &\Longrightarrow & j_\lambda  \leq \ell_\mu  
\ \ \ \ \ \ \ \ \ \ \ \ \ \ \ \ \ \ 
\mbox{for } 1 \leq \lambda  \leq r, \ 1 \leq \mu  \leq s.
\end{eqnarray}
The initial monomial
of
$
x_{M_{\soejia}} 
x_{M_{\soejib}} 
-
x_{M_{\soejic}} 
x_{M_{\soejid}} 
$
is 
$
x_{M_{\soejia}}x_{M_{\soejib}} 
$.
\end{Theorem}


\bigskip

\noindent
{\bf Example \ref{nestexample} (continued).}
The reduced Gr\"obner basis ${\cal G}$ in the statement
of Theorem \ref{main1} consists of 105 binomials.
For example, 
\begin{eqnarray*}
x_{ m_2^{(1)} m_2^{(1)}  }
x_{ m_2^{(2)} m_2^{(2)}  }
&-&
{x_{   m_2^{(1)} m_2^{(2)} }}^2
\\
x_{m_1^{(1)} m_1^{(1)}  }
x_{m_3^{(1)} m_2^{(2)} }
&-&
x_{m_1^{(1)} m_2^{(1)}  }
x_{m_2^{(1)} m_2^{(2)} }
\\
x_{ m_2^{(1)} m_2^{(1)}  }
x_{ m_1^{(1)} m_2^{(2)}  }
&-&
x_{ m_1^{(1)} m_2^{(1)}  }
x_{ m_2^{(1)} m_2^{(2)}  }
\\
{x_{   m_1^{(1)} m_2^{(1)} }}^2
&-&
x_{ m_1^{(1)} m_1^{(1)}  }
x_{ m_2^{(1)} m_2^{(1)}  }
\end{eqnarray*}
belong to ${\cal G}$
and the initial monomial is the first monomial for each binomial.

\medskip

\begin{proof}[Proof of Theorem \ref{main1}]
Let
$
x_{M_{\soejia}} 
x_{M_{\soejib}} 
-
x_{M_{\soejic}} 
x_{M_{\soejid}} 
\in {\cal G}$.
Thanks to the condition (1) above, we have
$\varphi_j(x_{M_{\soejia}} 
x_{M_{\soejib}}) - \varphi_j(x_{M_{\soejic}} 
x_{M_{\soejid}} ) \in I_{B_j}$
for all $1 \leq j \leq d$.
By virtue of Lemma \ref{keylemma}, we have $x_{M_{\soejia}} 
x_{M_{\soejib}} 
-
x_{M_{\soejic}} 
x_{M_{\soejid}} 
\in  I_{A(B_1,\ldots,B_d)}$.
Thus ${\cal G}$ is a subset of $I_{A(B_1,\ldots,B_d)}$.

Since the reduction relation modulo Gr\"obner bases are Noetherian,
\cite[Theorem 3.12]{Stu} guarantees that 
there exists a monomial order such that the monomial
$x_{M_{\soejia}} x_{M_{\soejib}} $ is the initial monomial for each 
$
x_{M_{\soejia}} 
x_{M_{\soejib}} 
-
x_{M_{\soejic}} 
x_{M_{\soejid}} 
\in {\cal G}$.

Suppose that there exists a binomial $0 \neq u - v \in 
I_{A(B_1,\ldots,B_d)}$
such that
neither $u$ nor $v$ is divided by the initial monomial of any binomial in ${\cal G}$.
By virtue of Lemma 1,
we have $\varphi_0(u) - \varphi_0(v) \in I_{A}$ and
$\varphi_i(u) - \varphi_i(v) \in I_{B_i}$ for all $1 \leq i \leq d$.
Hence $\varphi_i(u) - \varphi_i(v) \stackrel{{\cal G}_i}{\longrightarrow } 0$ 
for all $0 \leq i \leq d$.
Moreover, since neither $u$ nor $v$ is divided by the initial monomial of any binomial in ${\cal G}$,
we have 
$\varphi_i(u) \stackrel{{\cal G}_i}{\longrightarrow } \varphi_i(u)$ and
$\varphi_i(v) \stackrel{{\cal G}_i}{\longrightarrow } \varphi_i(v)$.
Thus $\varphi_i(u)= \varphi_i(v)$ for all $0 \leq i \leq d$.
By virtue of $\varphi_0(u)= \varphi_0(v)$ and our convention $(*)$,
\begin{eqnarray*}
u&=&
x_{M_{\ell_1}}
x_{M_{\ell_2}}
\cdots
x_{M_{\ell_p}},\\
v&=&
x_{M_{\ell_1'}}
x_{M_{\ell_2'}}
\cdots
x_{M_{\ell_p'}},
\end{eqnarray*}
where $1 \leq \ell_1 \leq \cdots \leq \ell_p \leq \eta$, $1 \leq \ell_1' \leq \cdots \leq \ell_p' \leq \eta$ and
\begin{eqnarray*}
M_{\ell_\xi  }&=&
{m}_{j_{\xi, 1}}^{(i_{\xi, 1})}{m}_{j_{\xi, 2}}^{(i_{\xi, 2})} \cdots {m}_{j_{\xi, r_\xi }}^{(i_{\xi, r_\xi })},
\\
M_{\ell_\xi'}&=&
{m}_{k_{\xi, 1}}^{(i_{\xi, 1})}{m}_{k_{\xi, 2}}^{(i_{\xi, 2})} \cdots {m}_{k_{\xi, r_\xi }}^{(i_{\xi, r_\xi })}.
\end{eqnarray*}
Since $\varphi_j(u)= \varphi_j(v)$ for all $1 \leq j \leq d$,
we have
$$
\prod_{i_{\xi, q} = j , \ 1 \leq \xi \leq p, \  1 \leq q \leq r_\xi}
z_{j_{\xi,q}}^{(i_{\xi, q})}
=
\prod_{i_{\xi, q} = j , \ 1 \leq \xi \leq p, \  1 \leq q \leq r_\xi}
z_{k_{\xi,q}}^{(i_{\xi, q})}.
$$
Thanks to the condition (2),
we have
$$
i_{\xi, q} = i_{\xi', q'} , \ \xi < \xi'
\Longrightarrow 
j_{\xi, q} \leq j_{\xi', q'},
$$
$$
i_{\xi, q} = i_{\xi', q'} , \ \xi < \xi'
\Longrightarrow 
k_{\xi, q} \leq k_{\xi', q'}.
$$
Hence 
$M_{\ell_\xi} = M_{\ell'_\xi}$
for all $1 \leq \xi \leq p$.
Thus we have $u =v$ and this is a contradiction.
\end{proof}

\begin{Example}
{\em
In the definition of a nested configuration, we assumed that
each $B_i$ and $B_j$ have no common variable.
If $B_i$ and $B_j$ have a common variable for some $1 \leq i < j \leq \lambda$,
then Theorem \ref{main1} does not hold in general.
For example,
if $A = \{t_1 t_4 , t_2 t_5, t_3 t_6\}$,
$B_1 = \{ u_1 \}$, $B_2 = \{ u_2 \}$, $B_3 = \{ u_3 \}$,
$B_4 = \{ v_1,v_2 \}$, 
$B_5 = \{ v_2,v_3 \}$ and
$B_6 = \{ v_1,v_3 \}$,
then
$I_{A(B_1,\ldots,B_6)}$ is a principal ideal
generated by a binomial of degree 3.
}
\end{Example}

\section{Nested configurations arising from Segre--Veronese configurations}

A typical class of semigroup rings
whose toric ideal possesses a quadratic initial ideal
is {\it algebras of Segre--Veronese type}
defined in \cite{OH3, aoki}.
Fix integers $\tau \geq 2$ and $n$ and
sets of integers
${\bf b} = \{b_1,\ldots,b_n\}$,
${\bf c} = \{c_1,\ldots,c_n\}$,
${\bf p} = \{p_1,\ldots,p_n\}$ and
${\bf q} = \{q_1,\ldots,q_n\}$
such that 
\begin{enumerate}
\item[(i)]
$0 \leq c_i \leq b_i$
for all $1 \leq i \leq n$;
\item[(ii)]
$1 \leq p_i \leq q_i \leq d$
for all $1 \leq i \leq n$.
\end{enumerate}
Let $\AAA \subset K[t_1,\ldots,t_d]$ denote the set
of all monomials
$\prod_{j=1}^{d} {t_j}^{f_j}$
such that
\begin{enumerate}
\item[(i)]
$\sum_{j=1}^{d} f_j =\tau $.
\item[(ii)]
$c_i \leq \sum_{j=p_i}^{q_i} f_j \leq  b_i$
for all $1 \leq i \leq n$.
\end{enumerate}
Then the affine semigroup ring $K[\AAA]$
is called an
{\it algebra of Segre--Veronese type}.

Several
popular classes of semigroup rings are
algebras of Segre--Veronese type.
If $n=2$, $\tau=2$, $b_1= b_2=c_1=c_2=1$,
$p_1 = 1$, $p_2 = q_1+1$ and $q_2 =d$, then
the affine semigroup ring
$K[\AAA]$
is
the Segre product of polynomial rings
$K[t_1,\ldots,t_{q_1}]$
and
$K[t_{q_1+1},\ldots,t_d]$.
On the other hand,
if $n=d$, $p_i=q_i=i$, $b_i=\tau$ and $c_i = 0$
for all $1 \leq i \leq n$, then the affine semigroup ring
$K[\AAA]$
is the classical
$\tau$th Veronese subring of the polynomial ring
$K[t_1,\ldots,t_d]$.
Moreover, 
if $n=d$, $p_i=q_i=i$, $b_i=1$ and $c_i = 0$
for all $1 \leq i \leq n$, then the affine semigroup ring
$K[\AAA]$
is the $\tau$th 
squarefree Veronese subring of the polynomial ring
$K[t_1,\ldots,t_d]$.
In addition, algebras of Veronese type 
(i.e., $n=d$, $p_i=q_i=i$ and $c_i = 0$
for all $1 \leq i \leq n$)
are 
studied in \cite{NeHi}
and \cite{Stu}.

Let $K[X]$ denote the polynomial ring with 
the set of variables 
$$
\left\{ x_{j_1 j_2 \cdots j_\tau}
\ \left| \  1\leq j_1 \leq  j_2 \leq \cdots \leq j_\tau \leq d, \ \prod_{k=1}^\tau t_{j_k} 
\in \AAA
\right. \right\}.
$$
The toric ideal $I_{\AAA}$
is the kernel of the surjective homomorphism
$
\pi : K[X]
\longrightarrow 
K[\AAA]
$
defined by
$
\pi(x_{j_1 j_2 \cdots j_\tau})
=
\prod_{k=1}^\tau t_{j_k}
$.
A monomial
$
x_{\ell_1 \ell_2 \cdots \ell_\tau}
x_{m_1 m_2 \cdots m_\tau}
\cdots
x_{n_1 n_2 \cdots n_\tau}
$
is called {\it sorted} if we have
$$
\ell_1 \leq m_1 \leq \cdots \leq n_1 \leq 
\ell_2 \leq m_2 \leq \cdots \leq n_2 \leq \cdots \leq 
\ell_\tau \leq m_\tau \leq \cdots \leq n_\tau
.$$
Let ${\rm sort}(\cdot)$ denote
the operator which takes any string over
the alphabet $\{1,2,\ldots,d\}$
and sorts it into weakly increasing order.

The squarefree quadratic Gr\"obner basis of 
the toric
ideal $I_\AAA$ is given as follows.

\begin{Theorem}[\cite{Stu, OH3, aoki}]
\label{sortingGB}
Work with the same notation as above.
Let ${\cal G}$ be the set of all binomials 
$$
x_{\ell_1 \ell_2 \cdots \ell_\tau}
x_{m_1 m_2 \cdots m_\tau}
-
x_{n_1 n_3 \cdots n_{2\tau-1}}
x_{n_2 n_4 \cdots n_{2\tau}}
$$
where
$$
{\rm sort}(
{\ell_1 m_1 \ell_2 m_2 \cdots \ell_\tau m_\tau}
)
=
n_1 n_2 \cdots n_{2\tau}.
$$
Then there exists a monomial order on $K[X]$
such that ${\cal G}$
is the reduced Gr\"obner basis of
the toric ideal $I_\AAA$.
The initial ideal is generated by 
squarefree quadratic (nonsorted) monomials.
\end{Theorem}

By virtue of Theorem \ref{main1},
if all of $A, B_1, \ldots , B_d$
are arising from Segre--Veronese configurations,
then the toric ideal of
the nested configuration
$ A(B_1,\ldots,B_d) $ 
possesses a quadratic 
Gr\"obner basis.
Although the following Gr\"obner basis is different 
from that Theorem \ref{main1} guarantee,
the proof is similar.

\begin{Theorem}
\label{main2}
If $K[A]$ is an algebra of Segre--Veronese type,
and if the toric ideals $I_{B_1}, \ldots , I_{B_d}$ possess 
the reduced 
Gr\"obner basis ${\cal G}_i$,
then the toric ideal $I_{A(B_1,\ldots,B_d) }$
possesses a 
quadratic 
Gr\"obner basis ${\cal G}$ consisting of all binomial of the form
$${\bf x}^\alpha -{\bf x}^\beta= 
x_{{m}_{j_1}^{(i_1)}{m}_{j_3}^{(i_3)} \cdots {m}_{j_{2r-1}}^{(i_{2r-1})}} 
x_{{m}_{j_2}^{(i_2)}{m}_{j_4}^{(i_4)} \cdots {m}_{j_{2r}}^{(i_{2r})}} 
-
x_{{m}_{\ell_1}^{(k_1)} {m}_{\ell_3}^{(k_3)} \cdots {m}_{\ell_{2r-1}}^{(k_{2r-1})}} 
x_{{m}_{\ell_2}^{(k_2)} {m}_{\ell_4}^{(k_4)} \cdots {m}_{\ell_{2r}}^{(k_{2r})}} 
$$
where ${\bf x}^\alpha$ is the initial monomial and
\begin{eqnarray}
k_1 k_2 \cdots k_{2r} &=& {\rm sort} (i_1 i_2 \cdots i_{2r} ),\\
\varphi_j({\bf x}^\alpha) &\stackrel{{\cal G}_j}{\longrightarrow } & \varphi_j({\bf x}^\beta) 
\ \ \ 
\mbox{ for each } 1 \leq j \leq d,\\
k_i = k_{i+1} &\Longrightarrow & \ell_i \leq \ell_{i+1} \ \ \ (1 \leq i \leq 2r -1).
\end{eqnarray}
Moreover, if the initial ideal of $I_{B_i}$ is squarefree for all $i$,
then the initial ideal of $I_{A(B_1,\ldots,B_d) }$ is squarefree.
\end{Theorem}

\begin{proof}
By virtue of Lemma \ref{keylemma}, ${\cal G}$ is a subset of $I_{A(B_1,\ldots,B_d)}$.
Since both the sorting operation and the reduction relation modulo Gr\"obner bases are Noetherian,
there exists a monomial order such that the first monomial ${\bf x}^\alpha$ is the initial monomial.

Suppose that there exists a binomial $0 \neq u - v \in 
I_{A(B_1,\ldots,B_d)}$
such that
neither $u$ nor $v$ is divided by any initial monomial of ${\cal G}$.
Let
\begin{eqnarray*}
u&=&
x_{{m}_{j_1}^{(i_1)}{m}_{j_{p+1}}^{(i_{p+1})} \cdots {m}_{j_{(r-1)p+1}}^{(i_{(r-1)p+1})}} 
x_{{m}_{j_2}^{(i_2)}{m}_{j_{p+2}}^{(i_{p+2})} \cdots {m}_{j_{(r-1)p+2}}^{(i_{(r-1)p+2})}} 
\cdots
x_{{m}_{j_p}^{(i_p)}{m}_{j_{2p}}^{(i_{2p})} \cdots {m}_{j_{r p},}^{(i_{r p})}} \\
v&=&
x_{{m}_{\ell_1}^{(k_1)}{m}_{\ell_{p+1}}^{(k_{p+1})} \cdots {m}_{\ell_{(r-1)p+1}}^{(k_{(r-1)p+1})}} 
x_{{m}_{\ell_2}^{(k_2)}{m}_{\ell_{p+2}}^{(k_{p+2})} \cdots {m}_{\ell_{(r-1)p+2}}^{(k_{(r-1)p+2})}} 
\cdots
x_{{m}_{\ell_p}^{(k_p)}{m}_{\ell_{2p}}^{(k_{2p})} \cdots {m}_{\ell_{r p}.}^{(k_{r p})}}
\end{eqnarray*}

By virtue of Lemma \ref{keylemma},
we have
$$
{\rm sort} (i_1 i_2 \cdots i_{r p} ) 
=
{\rm sort} (k_1 k_2 \cdots k_{r p}) 
.$$
Thanks to the condition (3), we have
$$
i_1 i_2 \cdots i_{r p}
=
{\rm sort} (i_1 i_2 \cdots i_{r p} ) 
=
{\rm sort} (k_1 k_2 \cdots k_{r p}) 
=
k_1 k_2 \cdots k_{r p}.
$$
Hence $i_q = k_q$ for all $1 \leq q \leq r p$.
By virtue of Lemma \ref{keylemma},
we have
$
\varphi_j(u) - \varphi_j(v) \in I_{B_j}
$
for each
$1 \leq i \leq d$.
Since each ${\cal G}_i$ consists of quadratic binomials and thanks to the condition (4),
$
\varphi_j(u) = \varphi_j(v) 
$
for each
$1 \leq i \leq d$.
Hence 
\begin{equation}
\prod_{i_q = j , \ \  1 \leq q \leq r p} z_{j_q}^{(i_q)}
=
\prod_{k_q = j , \ \  1 \leq q \leq r p} z_{\ell_q}^{(k_q)}
.
\end{equation}
Thanks to the condition (5) together with (6) above,
$j_q = \ell_q$ for all $1 \leq q \leq r p$.
Thus we have $u = v$ and this is a contradiction.

Suppose that
the initial ideal of $I_{B_i}$ is squarefree for all $i$ and
that
${x_{{m}_{j_1}^{(i_1)}{m}_{j_2}^{(i_2)} \cdots {m}_{j_{r}}^{(i_{r})}}}^2$
belongs to 
the initial ideal of $I_{A(B_1,\ldots,B_d) }$.
Then
$$
g = 
{x_{{m}_{j_1}^{(i_1)}{m}_{j_2}^{(i_2)} \cdots {m}_{j_{r}}^{(i_{r})}}}^2
-
x_{{m}_{\ell_1}^{(k_1)} {m}_{\ell_3}^{(k_3)} \cdots {m}_{\ell_{2r-1}}^{(k_{2r-1})}} 
x_{{m}_{\ell_2}^{(k_2)} {m}_{\ell_4}^{(k_4)} \cdots {m}_{\ell_{2r}}^{(k_{2r})}} 
$$
belongs to ${\cal G}$.
Thanks to the condition (3),
we have 
$$
g = 
{x_{{m}_{j_1}^{(i_1)}{m}_{j_2}^{(i_2)} \cdots {m}_{j_{r}}^{(i_{r})}}}^2
-
x_{{m}_{\ell_1}^{(i_1)} {m}_{\ell_3}^{(i_2)} \cdots {m}_{\ell_{2r-1}}^{(i_r)}} 
x_{{m}_{\ell_2}^{(i_1)} {m}_{\ell_4}^{(i_2)} \cdots {m}_{\ell_{2r}}^{(i_r)}}.
$$
Since $g \neq 0$, 
there exists $k$ such that $j_k \neq \ell_{2k}$.
Thanks to the condition (5),
we have
$$
\varphi_k
\left(
{x_{{m}_{j_1}^{(i_1)}{m}_{j_2}^{(i_2)} \cdots {m}_{j_{r}}^{(i_{r})}}}^2
\right)
\neq
\varphi_k
\left(
x_{{m}_{\ell_1}^{(i_1)} {m}_{\ell_3}^{(i_2)} \cdots {m}_{\ell_{2r-1}}^{(i_r)}} 
x_{{m}_{\ell_2}^{(i_1)} {m}_{\ell_4}^{(i_2)} \cdots {m}_{\ell_{2r}}^{(i_r)}}
\right)
.$$
Hence by the condition (4),
$$
\varphi_k
\left(
{x_{{m}_{j_1}^{(i_1)}{m}_{j_2}^{(i_2)} \cdots {m}_{j_{r}}^{(i_{r})}}}^2
\right)
=
\varphi_k
\left(
x_{{m}_{j_1}^{(i_1)}{m}_{j_2}^{(i_2)} \cdots {m}_{j_{r}}^{(i_{r})}}
\right)^2
$$
belongs to the initial ideal of $I_{B_k}$.
Since the initial ideal of $I_{B_k}$ is squarefree,
$$\varphi_k
\left(
x_{{m}_{j_1}^{(i_1)}{m}_{j_2}^{(i_2)} \cdots {m}_{j_{r}}^{(i_{r})}}
\right)
$$
belongs to the initial ideal of $I_{B_k}$.
This contradicts 
that $\varphi_k$ 
is defined with a standard expression.
Thus the initial ideal of $I_{A(B_1,\ldots,B_d) }$ is squarefree.
\end{proof}

\begin{Example}
{\em
Let $A = \{ t_1^2 \}$
and $B_1 = \{u_1  , u_2  , u_3 \}$.
Then 
$$A(B_1) = 
\{
u_1^2, u_1 u_2, u_1 u_3, u_2^2,u_2 u_3,u_3^2 
\}.$$
The reduced Gr\"obner basis in Theorem \ref{main1} consists of the binomials
\begin{eqnarray*}
x_{u_1 u_2}^2 &-& x_{u_1^2} x_{u_2^2}\\
x_{u_1 u_3}^2 &-& x_{u_1^2} x_{u_3^2}\\
x_{u_2 u_3}^2 &-& x_{u_2^2} x_{u_3^2}\\
 x_{u_1 u_2} x_{u_1 u_3} &-&  x_{u_1^2} x_{u_2 u_3}\\
x_{u_1 u_3} x_{u_2^2} &-& x_{u_1 u_2} x_{u_2 u_3}\\
x_{u_1 u_3} x_{u_2 u_3} &-& x_{u_1 u_2} x_{u_3^2}
\end{eqnarray*}
and the reduced Gr\"obner basis in Theorem \ref{main2} consists of the binomials
\begin{eqnarray*}
x_{u_1^2} x_{u_2^2} &-& x_{u_1 u_2}^2\\
x_{u_1^2} x_{u_3^2} &-& x_{u_1 u_3}^2\\
x_{u_2^2} x_{u_3^2} &-& x_{u_2 u_3}^2\\
 x_{u_1^2} x_{u_2 u_3} &-&  x_{u_1 u_2} x_{u_1 u_3}\\
x_{u_1 u_3} x_{u_2^2} &-& x_{u_1 u_2} x_{u_2 u_3}\\
 x_{u_1 u_2} x_{u_3^2} &-& x_{u_1 u_3} x_{u_2 u_3}
\end{eqnarray*}
where the initial monomial of each binomial is the first monomial.
}
\end{Example}

\noindent
Satoshi Aoki\\
Department of Mathematics and Computer Science,
Kagoshima University.\\
{\tt aoki@sci.kagoshima-u.ac.jp} \medskip\\ 
Takayuki Hibi\\
Graduate School of Information Science and Technology,
Osaka University.\\
{\tt hibi@math.sci.osaka-u.ac.jp} \medskip\\
Hidefumi Ohsugi\\
Department of Mathematics,
Rikkyo University.\\
{\tt ohsugi@rkmath.rikkyo.ac.jp} \medskip\\
Akimichi Takemura\\
Graduate School of Information Science and Technology,
University of Tokyo.\\
{\tt takemura@stat.t.u-tokyo.ac.jp}

\end{document}